\documentclass[letterpaper,11pt,reqno]{amsart}
\usepackage{graphicx,amsmath,amssymb,amsthm, float, comment}

\usepackage{srcltx}
\usepackage{mathrsfs}

\usepackage[utf8]{inputenc}
\usepackage[T1]{fontenc}

\usepackage[left=1in,right=1in, top=1in, bottom=1in]{geometry}
\usepackage[hidelinks]{hyperref}
\usepackage{xcolor}
\newtheorem{X}{X}[section]

\newtheorem{lemma}[X]{Lemma}

\newtheorem{proposition}[X]{Proposition}
\newtheorem{theorem}[X]{Theorem}

\def\SL{\operatorname{SL}}
\def\SO{\operatorname{SO}}
\renewcommand{\d}{{\rm d}}
\theoremstyle{definition}
\newtheorem{remark}[X]{Remark}

\newcommand{\RR}{\mathbb{R}}

\newcommand{\R}{\mathbb R}

\newcommand{\eps}{\varepsilon}

\newcommand{\sumstar}{\underset{\chi \bmod d}{{\sum}^*}}

\numberwithin{equation}{section}

\title[Low-lying zeros in families of Maass form $L$-functions]{Low-lying zeros in families of Maass form $L$-functions:\\ an extended density theorem}

\author{Martin \v Cech, Lucile Devin, Daniel Fiorilli, Kaisa Matom\"aki and Anders S\"odergren}
\address{Charles University, Faculty of Mathematics and Physics, Department of Mathematical Analysis\\
\rule[0ex]{0ex}{0ex}\hspace{8pt} and Department of Algebra, Sokolovsk\'{a} 83, 18600 Praha 8, Czech Republic}
\email{martin.cech@matfyz.cuni.cz}
\address{Univ. Littoral C\^ote d'Opale, UR 2597
	LMPA, Laboratoire de Math\'ematiques Pures et Appliqu\'ees \newline
	\rule[0ex]{0ex}{0ex}\hspace{8pt} Joseph Liouville,
	F-62100 Calais, France}
\email{lucile.devin@univ-littoral.fr}
\address{CNRS, Universit\'e Paris-Saclay, Laboratoire de math\'ematiques d'Orsay, 91405, Orsay, France}
\email{daniel.fiorilli@universite-paris-saclay.fr}
\address{Department of Mathematics and Statistics, University of Turku, 20014 Turku, Finland}
\email{ksmato@utu.fi}
\address{Department of Mathematical Sciences, Chalmers University of Technology and the University \newline
\rule[0ex]{0ex}{0ex}\hspace{8pt} of Gothenburg, SE-412 96 Gothenburg, Sweden}
\email{andesod@chalmers.se}

\date{\today}
\begin{document}

\begin{abstract}
We study the one-level density of low-lying zeros in the family of Maass form $L$-functions of prime level~$N$ tending to infinity.
Generalizing the influential work of Iwaniec, Luo and Sarnak to this context, Alpoge et al. have proven the Katz--Sarnak prediction for test functions whose Fourier transform is supported in $(-\frac32,\frac32)$. In this paper, we extend the unconditional admissible support to $(-\frac{15}8,\frac{15}8)$.
The key tools in our approach are analytic estimates for integrals appearing in the Kutznetsov trace formula, as well as a reduction to bounds on Dirichlet polynomials, which eventually are obtained from the large sieve and the fourth moment bound for Dirichlet $L$-functions. Assuming the Grand Density Conjecture, we extend the admissible support to $(-2,2)$. 
In addition, we show that the same techniques also allow for an unconditional improvement of the admissible support in the corresponding family of $L$-functions attached to holomorphic forms.
\end{abstract}
\maketitle

\section{Introduction and statement of results}

Since the work of Riemann \cite{Ri}, the study of the distribution of non-trivial zeros of $L$-functions has been a central question in number theory, due to their strong influence on prime numbers, number fields, elliptic curves, exponential sums and many other arithmetic objects.
In particular, the Generalized Riemann Hypothesis (GRH), stating that their real part is equal to $\frac12$, is one of the key open problems in modern mathematics.
Beyond their horizontal distribution, much attention has been directed towards understanding the vertical distribution, leading to interesting connections with random matrix theory (see e.g. \cite{Mo1, Hej, RS1, RS, Odl}).

In a similar direction, Katz and Sarnak \cite{KS} conjectured that the vertical distribution of low-lying zeros of $L$-functions in a family is governed by the distribution of eigenvalues of matrices in suitable compact Lie groups (see also the recent account \cite{SST}). As usual, by low-lying zeros we mean zeros with small imaginary part. One of the landmark papers in this topic is \cite{ILS}, in which families of $L$-functions attached to holomorphic modular forms and their symmetric squares were studied in the weight and level aspects. 
Refinements of this work include computations of lower-order terms~\cite{M, MM, RR1, DFS1}, as well as improvements of the admissible support~\cite{RR2, FKM, DFS2, BCL}.
These results have also been complemented by the works~\cite{AM, AAILMZ, LQ, MT} for $L$-functions attached to Maass forms, again in both the level and eigenvalue aspects.

The goal of the current paper is to extend the admissible support in the level aspect for both Maass form and holomorphic form families. Since our improvement of the support is more substantial in the Maass form case, this will be our main focus. 

\subsection{Maass forms}\label{ssec:Maass}
In order to state our results, we need to introduce some notation; for background on Maass forms and their $L$-functions, see e.g.~\cite{DFI, Iw2, KL}.
We let  $B^*(N)$ denote an orthogonal basis in the space of Maass cusp newforms of prime level~$N$ (with trivial nebentypus) which are eigenfunctions of the hyperbolic Laplacian, all Hecke operators, and the reflection operator sending a function $f(z)$ to $f(-\bar z)$. We denote the corresponding eigenvalues by $\lambda_f=\frac14+t_f^2$, $\lambda_f(n)$, $n\geq1$, and $\epsilon_f$, respectively. 
We write the Fourier expansion of $f\in B^*(N)$ at infinity as
\begin{align*}
f(x+iy)=\sqrt{y}\sum_{n\neq0}\rho_f(n)K_{it_f}(2\pi |n|y)e^{2\pi i nx},
\end{align*}
where $K_{\nu}(x)$ is the $K$-Bessel function, and we normalize so that $\rho_f(1) =1$, see e.g.~\cite[\S 4.1]{KL}. We recall the important relation $\rho_f( n)=\lambda_f(n)$, $n\geq 1$, between the Fourier coefficients and the Hecke eigenvalues of~$f$ (see \cite[\S 4.2]{KL}).  The $L$-function associated with $f$ is given by 
\begin{align}\label{Eq L-function}
L(s,f)=\sum_{n=1}^{\infty}\lambda_f(n)n^{-s}=\prod_{p} \Big( 1-\frac {\alpha_f(p)}{p^s}\Big)^{-1}\Big( 1-\frac {\beta_f(p)}{p^s}\Big)^{-1} = \prod_p \Big(1-\frac{\lambda_f(p)}{p^s} + \frac{1}{p^{2s}}\Big)^{-1}
\end{align}
for $\Re(s)>1$ and is known to extend to an entire function of $s$. We recall that the Kim--Sarnak bound \cite{KiS} gives
\begin{align}
  |\alpha_f(p)|,|\beta_f(p)|\leq p^{\frac7{64}}.
  \label{equation Kim-Sarnak}
\end{align}
Note also that $\beta_f(p) =0$ when $p\mid N$. Moreover, we have the functional equation 
\begin{align}\label{align:FE}
\Lambda_f(s):=\Big(\frac{N^{\frac 12}}\pi\Big)^{s}\Gamma\Big(\frac{s+\delta_f+it_f}{2}\Big)\Gamma\Big(\frac{s+\delta_f-it_f}{2}\Big)L(s,f)=\epsilon_f\Lambda_f(1-s),
\end{align}
where $\delta_f=0$ if $\epsilon_f=1$ (i.e. if $f$ is even) and $\delta_f=1$ if $\epsilon_f=-1$ (i.e. if $f$ is odd); see e.g.~\cite[\S 8]{DFI}. 

For technical reasons, we next introduce a class $\mathscr H(A,\delta)$ of weight functions which are suitable for application in the Kuznetsov trace formula (see, e.g., \cite[Chapter 8]{KL}). Here, $A>13$ and $\delta>0$ are fixed numbers, and the elements in $\mathscr H(A,\delta)$ are nonnegative even functions $h:\RR\to\RR$ which are not identically zero and extend to even holomorphic functions on the horizontal strip
$\{s : |\Im(s)|\leq A\}$ satisfying, uniformly for $|\Im(s)| \leq A$, 
\begin{align*}
h(s)\ll (1+|s|)^{-2-\delta}.
\end{align*}

In this paper, we will be interested in the one-level density of low-lying zeros in the family of $L$-functions $L(s,f)$ for $f\in B^*(N)$. First, for $h\in\mathscr H(A,\delta)$, we introduce the notation $$ \Omega^*(h,N):=\sum_{f\in B^*(N)}\omega_f(N) h(t_f)$$ for the total (weighted) mass of $h$, where  $\omega_f(N)$ is the spectral weight of $f$, which we define to be
$$\omega_f(N):=\frac{1}{\lVert f \rVert_N^2\cosh(\pi t_f)}.$$
Here, the Petersson norm is defined as
$$\lVert f \rVert_N^2 := \int_{\Gamma_0(N) \backslash \mathbb H} |f(x+iy)|^2 \frac{\d x \d y }{y^2}.  $$
Combining \cite{Iw} and \cite{GHL}, we also note that, for any $\varepsilon > 0$,
 $$ N ^{-1-\eps} \ll_{t_f,\eps} \omega_f(N) \ll_{t_f,\eps} N^{-1+\eps}.$$ Now, we let $\phi :\mathbb R \rightarrow \mathbb R$ be an even Schwartz test function, that is, $\phi$ is smooth and such that $(1+|x|)^\ell \phi^{(k)}(x) \ll_{\ell, k} 1$ for every $k, \ell \in \mathbb{N}$. Note that we may holomorphically extend $\phi$ to the whole complex plane
(see, e.g., [R, Section 19.1]).

Finally, for a fixed function $h\in \mathscr{H}(A,\delta)$, we define the one-level density
\begin{align}\label{weighted 1 level density}\mathcal D^*(\phi,h;N):= \frac 1{\Omega^*(h,N)}\sum_{f \in B^*(N)} \omega_f(N) h(t_f) \sum_{\gamma_f} \phi\Big( \gamma_f \frac{\log N}{2\pi} \Big),
\end{align}
where $\rho_f = \frac 12+i\gamma_f$ runs through the non-trivial zeros of $L(s, f)$ (we stress that $\gamma_f$ need not be real). Note that we normalize by $\log N$ instead of the logarithm of the analytic conductor (which is approximately $t^2_fN$; see \cite[Section 5.11]{IK}), since these two quantities are asymptotically equivalent for $f$ in the bulk of the outer sum in \eqref{weighted 1 level density} (as $h(t_f)$ essentially restricts $t_f$ to be small).
In the following, we will allow all implicit constants to depend on $h$.

For the present family of Maass form $L$-functions, the Katz--Sarnak prediction (see~\cite{KS} and~\cite[Conjecture 2]{SST}) is that $\mathcal D^*(\phi,h;N)$ converges, as $N\to\infty$, to the random matrix integral $$\int_{\mathbb R} W(O)(x) \phi(x) \d x,$$ where $W(O)(x)=1+\frac 12\delta_0(x)$ is the limit density of rescaled eigenvalues close to $1$ for large random orthogonal matrices. This prediction has been confirmed by Alpoge, Amersi, Iyer, Lazarev, Miller, and Zhang \cite{AAILMZ} for test functions $\phi$ whose Fourier transform $\widehat \phi$ satisfy ${\rm supp}(\widehat \phi)\subset(-\frac32,\frac32)$. Extending the admissible support in density results is one of the central problems in this subject, due in particular to important arithmetic consequences such as bounds on the average order of vanishing at the central point; see, e.g., \cite{ILS,Y} as well as the more recent papers \cite{FKM,DPR,DFS2,BCL}. In this paper, we prove such a result: our main theorem shows that the Katz--Sarnak prediction for the current family holds in the extended range ${\rm supp}(\widehat \phi)\subset (-\frac{15}8 ,\frac{15}8)$.

\begin{theorem}\label{theorem main}
Let $\phi$ be an even Schwartz function for which  ${\rm supp}(\widehat \phi) \subset (-\frac{15}8 ,\frac{15}8)$. Then, for any fixed $h\in \mathscr{H}(A,\delta)$ and for $N$ running through the set of prime numbers, we have the estimate 
\begin{equation}
\mathcal D^*(\phi,h;N) =  \int_{\mathbb R} W(O)(x) \phi(x) \d x  +o_{N\rightarrow \infty}(1).
\label{equation main theorem}
\end{equation}
\end{theorem}

Let us give a brief outline of the proof of Theorem~\ref{theorem main}. We first apply the explicit formula
and express the one-level density $\mathcal{D}^\ast(\phi, h; N)$ as a sum over eigenvalues of Hecke operators at prime
power values. Averaging over the family of newforms of prime level, we apply the Kuznetsov formula
and turn this last expression into a weighted sum of Kloosterman sums. The next step is to rewrite the Kloosterman sums in terms of Dirichlet characters and Gauss sums using orthogonality.  After applying Mellin inversion, we end up having to bound a weighted version of the following average of Dirichlet polynomials: 
$$\int_{-\infty}^\infty\sumstar\left|\sum_{n \leq N^{\frac{15}8-\varepsilon}} \frac{\Lambda(n) \chi(n)}{ n^{ \frac 12+it}}\right|\frac{\d t}{t^2+1}. $$
At this point, we apply Heath-Brown's identity (actually, for technical reasons, we apply it already before Mellin inversion) to decompose $\Lambda(n)$ as a sum of various convolutions. Then we obtain the desired bound by using H\"older's inequality, the large sieve and the fourth moment bound for Dirichlet $L$-functions.

In an earlier paper \cite{DFS2}, the second, third, and fifth authors proved a result analogous to Theorem~\ref{theorem main} in the family of holomorphic cusp forms in the level aspect using zero-density estimates for Dirichlet $L$-functions instead of applying Heath-Brown's identity (naturally other details of the proof are different as well, such as using the Peterson formula instead of the Kuznetsov formula). The zero density estimates could be used in the Maass form case as well, however the resulting admissible range for  ${\rm supp}(\widehat \phi)$ would be $(-1-\frac{\sqrt{3}}{2}, 1+\frac{\sqrt{3}}{2})$, which is slightly narrower than the range in Theorem~\ref{theorem main}. 
See Remark~\ref{remark HB beats zeros} for more on the relation between the two approaches.

It turns out that in the holomorphic case, one can also apply Heath-Brown's identity and Dirichlet polynomial estimates to obtain a larger support; we will state our theorem in this context in the following subsection.

\subsection{Holomorphic forms} \label{ssec:introholom}
In this subsection we recall the notation of~\cite{DFS2}.
We fix a basis $B_k^\ast(N)$ of Hecke eigenforms of the space $H_k^\ast(N)$ of newforms of prime level $N$ and weight $k$. We normalize so that, for every
\begin{align*}
  f(z) = \sum_{n = 1}^\infty \lambda_f(n) n^{\frac{k-1}{2}} e^{2\pi i n z} \in B_k^\ast(N),
  \end{align*}
  we have $\lambda_f(1) = 1$. We use the harmonic weights defined as
  \begin{align*}
\omega_{f,k}(N) := \frac{\Gamma(k-1)}{(4\pi)^{k-1}(f, f)_{k,N}},
  \end{align*}
  with
  \begin{align*}
(f, f)_{k,N} := \int_{\Gamma_0(N) \setminus \mathbb{H}} y^{k-2} |f(z)|^2 \d x \d y.
  \end{align*}
  The object we are interested in is the one-level density
  \begin{align*}
    \mathcal{D}_{k, N}^\ast(\phi; X) := \frac{1}{\Omega_k(N)} \sum_{f \in B_k^\ast(N)} \omega_{f,k}(N) \sum_{\gamma_f} \phi\left(\gamma_f \frac{\log X}{2\pi}\right),
  \end{align*}
  where $\rho_f = \frac 12+i\gamma_f$ runs through the non-trivial zeros of $L(s, f)$, $\phi$ is an even Schwartz function whose Fourier transform is compactly supported, $X = k^2N$ is the size of the analytic conductor of $L(s, f)$, and the total weight is
  \begin{align*}
    \Omega_k(N) := \sum_{f \in B_k^\ast(N)} \omega_{f,k}(N).
    \end{align*}
    Writing
    \begin{equation}\label{eq:Thetakdef}
      \Theta_k := 2-\frac{1}{5k-2},
    \end{equation}
    we obtain the following theorem.

\begin{theorem}\label{theorem main holom}
Let $k \in \mathbb{N}$ be even and let $\phi$ be an even Schwartz function for which ${\rm supp}(\widehat \phi) \subset (-\Theta_k, \Theta_k)$. Then, for $N$ running through the set of prime numbers, we have the estimate 
\begin{equation*}
  \mathcal D_{k, N}^\ast(\phi,k^2 N) =  \int_{\mathbb R} W(O)(x) \phi(x) \d x  +o_{N\rightarrow \infty}(1).
\end{equation*}
\end{theorem}
This improves upon the previous work of the second, third, and fifth authors~\cite[Theorem 1.1]{DFS2} which required ${\rm supp}(\widehat \phi) \subset (-\Theta_k', \Theta_k')$ with
\[
  \Theta_k' :=
  \begin{cases}
    1+\frac{\sqrt{3}}{2} & \text{if $k = 2$;} \\
    2 - \frac{1}{5k-\frac{5}{2}} & \text{if $k \geq 4$}.
    \end{cases}
\] 
\subsection{A conditional result}
Now we return to the Maass form set-up from Section~\ref{ssec:Maass}. In the next result, we show that under a classical zero-density conjecture (which is significantly weaker than the GRH), the admissible support of $\widehat \phi$ can be extended to the range $(-2,2)$. Iwaniec and Kowalski~\cite[(10.7)]{IK} have coined the term Grand Density Conjecture for the following claim: For every $\beta \in [1/2, 1]$, $Q \geq 1$, and $k \in \mathbb{N}$, we have 
\begin{equation}
\sum_{\substack{q\leq Q \\ (q,k)=1}}\underset{\psi \bmod q}{{\sum}^*} \hspace{.2cm} \sum_{\xi \bmod k} N(\beta,T,\xi\psi) \ll_\eps (kQ^2T)^{2(1-\beta)}(\log kQT)^{O(1)},
\label{bound grand density}
\end{equation}  
where for a Dirichlet character $\chi\bmod q$,
$$N(\beta,T,\chi):= \#\lbrace \rho_\chi \in \mathbb{C} : \Re(\rho_\chi) \geq  \beta, \lvert\Im(\rho_\chi)\rvert \leq T , L(\rho_\chi,\chi) =0\rbrace,$$
and the star on the sum means that the sum runs over primitive characters. 

\begin{theorem}
\label{theorem second}
Let $\phi$ be an even Schwartz function for which  ${\rm supp}(\widehat \phi) \subset (-2 ,2)$, and assume the Grand Density Conjecture~\eqref{bound grand density}.
Then~\eqref{equation main theorem} holds.
\end{theorem}
An analogous result also holds for holomorphic forms (see~\cite[Remark 4.2]{DFS2}).
We will deduce Theorem~\ref{theorem second} in the end of Section~\ref{sec:Kuznetsov} by applying similar techniques combined with analytic estimates on integral transforms appearing in the Kuznetsov formula (that we use also in the proof of Theorem~\ref{theorem main}).

\section*{Acknowledgments}

We are grateful to Didier Lesesvre and Morten Risager for helpful discussions. This material is based upon work partially supported by the Swedish Research Council under grant no.\ 2021-06594 while the authors were in residence at the Institut Mittag--Leffler in Djursholm, Sweden during the spring semester of 2024. The first author was supported by the Research Council of Finland grant no. 333707 and by the Charles University grants PRIMUS/24/SCI/010 and PRIMUS/25/SCI/017. The second author was partially supported by the grant KAW 2019.0517 from the Knut and Alice Wallenberg Foundation and by the PEPS JCJC 2023 program of the INSMI (CNRS). The fourth author was partially supported by the Research Council of Finland grants no. 333707 and 346307. The fifth author was supported by the grant 2021-04605 from the Swedish Research Council. We also thank the Anna-Greta and Holger Crafoord Fund and the Royal Swedish Academy of Sciences for supporting this project via the grant CRM2020-0008, as well as the IHP Research in Paris program for providing funding and excellent working conditions.

\section{The Kuznetsov trace formula and analytic estimates}\label{sec:Kuznetsov}

We work in the set-up of Section~\ref{ssec:Maass}. We recall that $B^*(N)$ is an orthogonal basis of Maass cusp newforms of level $N$. The following lemma gives an explicit formula for forms in $B^\ast(N)$, relating a sum over zeroes of the corresponding $L$-function to a sum of Hecke eigenvalues at prime powers.

\begin{lemma}\label{Lemma Explicit Formula}
Let $N$ be a prime, let $f\in B^*(N)$ be a Hecke-Maass newform, and let $\phi$ be an even Schwartz test function whose Fourier transform has bounded support. Then

\begin{equation}
\sum_{\gamma_f}\phi\Big(\gamma_f\frac{\log N}{2\pi}\Big)=\widehat \phi(0)  +\frac{\phi(0)}2 -2\sum_{\substack{p,\nu  }} \frac{\lambda_f(p^{\nu}) }{p^{\frac \nu 2}} \widehat \phi \Big( \frac{\nu \log p}{\log N} \Big) \frac{\log p}{\log N}
+O\Big(\frac {\log (|t_f|+2)}{\log N}\Big).
\label{equation lemma explicit formula}    
\end{equation}
Moreover, if $f\in B^*(1)$ and $X\geq 2$ is a real parameter, then 
\begin{equation}
\sum_{\gamma_f}\phi\Big(\gamma_f\frac{\log X}{2\pi}\Big)=\frac{\phi(0)}2-2\sum_{\substack{p,\nu  }} \frac{\lambda_f(p^{\nu}) }{p^{\frac \nu 2}} \widehat \phi \Big( \frac{\nu \log p}{\log X} \Big) \frac{\log p}{\log X}+O\Big(\frac {\log (|t_f|+2)}{\log X}\Big).
\label{equation lemma explicit formula 2}    
\end{equation}
\end{lemma}

\begin{proof}
We will apply~\cite[Proposition 2.1]{RS} with $ g(u):= \frac 1 {\log N} \widehat \phi(\frac u{\log N}) $ so that $h(r)=\phi(r \frac{\log N}{2\pi } ) $; by the functional equation~\eqref{align:FE}, we can take $\mu_\pi(1)= \delta_f+it_f$, $\mu_\pi(2)= \delta_f-it_f$, and moreover we have $\overline t_f \in \{ t_f, -t_f\}$. It follows that the left-hand side of~\eqref{equation lemma explicit formula} is equal to
\begin{multline*}  
\widehat{\phi}(0)\Big(1 - \frac{2\log \pi}{\log N}\Big) -2\sum_{p,\nu} \frac{\alpha_f^{\nu}(p)+\beta_f^{\nu}(p) }{p^{\frac \nu 2}} \widehat \phi \Big( \frac{\nu \log p}{\log N} \Big) \frac{\log p}{\log N} \\ + \frac 1{\log N} \int_{\mathbb R} \Big(  \frac{\Gamma'}{\Gamma} \Big( \frac 14+ \frac{\delta_f+it_f}2+\frac{\pi i u}{\log N} \Big)+ \frac{\Gamma'}{\Gamma} \Big( \frac 14 + \frac{\delta_f-it_f}2+\frac{\pi i u}{\log N} \Big)\Big) \phi(u) \,\d u,
\end{multline*}
where $\alpha_f(p)$ and  $\beta_f(p)$ are the local coefficients of $L(s,f)$ defined in \eqref{Eq L-function}. Furthermore, Stirling's formula implies that the third term in this expression is $ \ll_\phi \frac {\log(|t_f|+2)}{\log N}$.

Note that $\lambda_f(p^\nu) =\alpha_f^\nu(p) = \alpha_f^\nu(p) + \beta_f^\nu(p)$ for $p\mid N$, and for $p \nmid N$ we obtain from~\eqref{Eq L-function} that
\begin{align*}
\begin{aligned}
  \lambda_f(p^\nu) &= \sum_{\substack{k, \ell \geq 0 \\ k+\ell = \nu}} \alpha_f^k(p) \beta_f^\ell(p) = \alpha_f^\nu(p) + \beta_f^\nu(p) + \alpha_f(p) \beta_f(p) \sum_{\substack{k, \ell \geq 0 \\k+\ell = \nu-2}} \alpha_f^k(p) \beta_f^\ell(p) \\
  &= \alpha_f^\nu(p) + \beta_f^\nu(p) + \lambda_f(p^{\nu-2}).
\end{aligned}
\end{align*}
Hence
\begin{align}\label{align:primesum}
  \begin{aligned}
-&2\sum_{p,\nu} \frac{\alpha_f^{\nu}(p)+\beta_f^{\nu}(p) }{p^{\frac \nu 2}} \widehat \phi \Big( \frac{\nu \log p}{\log N} \Big) \frac{\log p}{\log N} \\
&=-2\sum_{\substack{p,\nu }} \frac{\lambda_f(p^{\nu}) }{p^{\frac \nu 2}} \widehat \phi \Big( \frac{\nu \log p}{\log N} \Big) \frac{\log p}{\log N}+2\sum_{\substack{p,\nu \\ p\nmid N  \\ \nu\geq 2}} \frac{\lambda_f(p^{\nu-2}) }{p^{\frac \nu 2}} \widehat \phi \Big( \frac{\nu \log p}{\log N} \Big) \frac{\log p}{\log N}.
  \end{aligned}
  \end{align}
We can discard the terms with $\nu \geq 3$ in the second sum since by the Kim--Sarnak bound their contribution is 
$$\ll \frac 1{\log N}  \sum_{\substack{p \nmid N \\ \nu\geq 3}} \frac{\log p}{p^{\nu(\frac 12-\frac 7{64}) + \frac{7}{32} - \eps}} \ll \frac 1{\log N}.$$
As for the terms with $\nu=2$, by the Prime Number Theorem, their contribution in the second sum is equal to
\begin{align*}
 &\frac{2}{\log N}\sum_{p \nmid N} \frac{\log p}{p} \widehat \phi \Big( \frac{2 \log p}{\log N} \Big)= \frac{\phi(0)}2 +O\Big( \frac 1{\log N}\Big).
\end{align*}
Thus~\eqref{align:primesum} equals
\begin{align*}
  & -2\sum_{\substack{p,\nu  }} \frac{\lambda_f(p^{\nu}) }{p^{\frac \nu 2}} \widehat \phi \Big( \frac{\nu \log p}{\log N} \Big) \frac{\log p}{\log N} +\frac{\phi(0)}2+O\Big(\frac 1{\log N}\Big),
\end{align*}
and~\eqref{equation lemma explicit formula} follows. The proof of~\eqref{equation lemma explicit formula 2} is similar.
\end{proof}

To evaluate the one-level density~\eqref{weighted 1 level density} using this explicit formula, we will sum the right-hand side of~\eqref{equation lemma explicit formula} over $f\in B^*(N)$ against the weight $\omega_f(N)h(t_f)$. This will involve the quantity
$$\Delta^*_N(m,n;h):=\sum_{f\in B^*(N)}\omega_f(N) h(t_f)\lambda_f(m)\lambda_f(n),$$
which we need to estimate. In order to do so, we will apply the standard Kuznetsov trace formula, which gives a formula for the closely related quantity
$$\Delta_N(m,n;h):=\sum_{f\in B(N)}\omega_f(N) h(t_f)\lambda_f(m)\lambda_f(n),$$
where $B(N)$ is any orthogonal basis of the full space of Maass cusp forms of level $N$.
The following lemma allows us to express $\Delta^*_N(m,n;h)$ in terms of $\Delta_N(m,n;h)$ ($N$ being prime makes the expression particularly simple).

\begin{lemma}
\label{corollary kim sarnak}
Let $h\in \mathscr{H}(A,\delta)$ and let $N$ be a prime. For positive integers $m$ and $n$, with $(mn,N)=1$, we have
\begin{equation*}
		\Delta^*_N(m,n;h) =\Delta_N(m,n;h)-\frac 2{N+1}\Delta_1(m,n;h).
	\end{equation*}
\end{lemma}

\begin{proof}  
The main idea of the proof, inspired by~\cite[Proposition 2.6]{ILS}, is to make a specific choice of the basis $B(N)$ satisfying $B^*(N)\subset B(N)$.
Starting from the decomposition of the space of Maass cusp forms of level $N$
$$\bigoplus_{LM=N}\bigoplus_{f\in B^*(M)}\text{span}\{f_{|\ell } : \ell\mid L \},$$
where $f_{|\ell }(z):=f(\ell z)$ (see, e.g., \cite[Theorem 4.6]{St}), we will use the Gram-Schmidt process. Indeed, assuming that $(n,N)=1$, the $n$-th Hecke operator as well as the hyperbolic Laplacian commute with the operator $f\mapsto f_{|N}$. 
Hence, for each $f \in B(1)=B^*(1)$, one can construct $f_N$ in $\text{span}\{f_{|1 }, f_{|N} \}$ orthogonal to $f_{|1}= f$ and satisfying
\begin{align*}
\lVert f_N \rVert_N = \lVert f \rVert_N, \qquad \lambda_{f_N}(n)=\lambda_{f}(n)\qquad\text{and}\qquad t_{f_N}=t_f.
\end{align*}
In particular, using the relations $\lVert f_N \rVert_N^2 = \lVert f \rVert_N^2 = (N+1)\lVert f \rVert_1^2 $, we note that
$$ \omega_{f_N}(N):=\frac{1}{\lVert f_N \rVert_N^2\cosh(\pi t_{f_N})}=\frac{\omega_{f}(1)}{N+1} \qquad \text{and}\qquad  \omega_{f}(N)=\frac{\omega_{f}(1)}{N+1}.$$
We thus define the orthogonal basis $B(N) := B^{*}(N) \cup \bigcup\limits_{f \in B(1)}\lbrace f,f_{N} \rbrace$.
From the above observations, we get
\begin{align*}
\Delta_N(m,n;h)&=\sum_{f \in B(N)}\omega_f(N) h(t_f)\lambda_f(m)\lambda_f(n)\\
&=\sum_{f\in B^*(N)}\omega_{f}(N)h(t_f)\lambda_f(m)\lambda_f(n)
+
2\sum_{f\in B(1)}\frac{\omega_{f}(1)}{N+1}h(t_f)\lambda_f(m)\lambda_f(n)\\
&= \Delta^*_N(m,n;h) + \frac{2}{N+1}\Delta_1(m,n;h),
\end{align*}
which is the desired result.
\end{proof}

Let us now state the Kuznetsov trace formula. It will involve the $J$-Bessel function
\[
J_{\nu}(x) :=  \sum_{m=0}^{\infty} \frac{(-1)^m}{m!\Gamma(m+\nu+1)} \Big( \frac x2\Big)^{2m+\nu},
\]
the integral transformation
\begin{equation*}
H^{+}(x):=\frac{2i}{\pi}\int_{-\infty}^{\infty}h(t)J_{2it}(x)\frac{t}{\cosh(\pi t)}\,\d t
\end{equation*}
of $h \in \mathscr{H}(A,\delta),$ and the arithmetic function 
$$\sigma_{\alpha}(n)=\sum_{d|n}d^{\alpha}.$$
\begin{lemma}\label{Kuznetsov1}
Let $h\in \mathscr{H}(A,\delta)$ and let $N$ be a prime. For positive integers $m$ and $n$, with $(mn,N)=1$, we have
\begin{align*}
\begin{aligned}
  \Delta_N(m,n;h)=&\frac{\delta_{m,n}}{\pi^2}\int_{-\infty}^{\infty}th(t)\tanh(\pi t)\,\d t
-\frac{1}{\pi N}\Big(1 + \frac1N\Big)\int_{-\infty}^{\infty}h(t)\frac{\sigma_{2it}(m)\sigma_{2it}(n)}{(mn)^{it}|\zeta_N(1+2it)|^2}\,\d t\nonumber\\
&+\sum_{ c\equiv 0 \bmod N}c^{-1}S(m,n;c)H^{+}\left(\frac{4\pi\sqrt{mn}}{c}\right),
\end{aligned}
\end{align*}
where
 $\zeta_N(s) = \prod_{p\nmid N} (1 - p^{-s})^{-1}.$
\end{lemma}

\begin{proof}
This is \cite[Theorem~7.14]{KL} applied with $\bold{\texttt{n}} =1$ and trivial nebentypus $\omega'$, see~\cite[(7.32)]{KL}. Note that in \cite{KL}, the Petersson norm has a different normalization from ours (see \cite[(4.3)]{KL}). In our case, the contribution of the Eisenstein term has two parts: in the first part, we have $i_N =0$, $\chi_1'$ equal to the trivial character modulo~$N$ and $\chi_2'$ equal to the trivial character modulo~$1$, while in the second part we have $i_N = 1$, $\chi_1'$ equal to the trivial character modulo~$1$ and $\chi_2'$ equal to the trivial character modulo~$N$ (with notation as in \cite{KL}). Note also that the hypothesis on $h$ is weaker in our statement compared to \cite[Theorem~7.14]{KL}; however, this is justified in \cite[Theorem~8.1]{KL} and the following remarks.
\end{proof}

In order to successfully apply the Kuznetsov formula, we need to understand the behavior of $H^+$. This is the aim of our next lemma.

\begin{lemma}
\label{lemma bound H+ small x}
Let $h\in \mathscr{H}(A,\delta)$ for some $A>13$. Then $H^+(x)$ is well-defined and twice differentiable in the range $ 0 < x < 4\pi$,
and moreover we have the estimate
\begin{equation}
 H^{+}(x) =  \frac{4}{\pi}\sum_{k=0}^{\lfloor A \rfloor-1} (-1)^k (k+\tfrac 12)h(-i(k+\tfrac 12)) J_{2k+1}(x) + O_A(x^ {2\lfloor A \rfloor }).
 \label{equation analytic continuation H+}
\end{equation}
In particular, $\lim_{x\rightarrow 0}H^+(x)=0$, $H^+(x) \ll x$, 
and more generally we have the Taylor expansion\begin{align} \label{eq:TaylorExp}
  H^{+}(x) = \sum_{m=1}^{2\lfloor A \rfloor-1} \alpha_{m,h} x^{m} +O_{A}(x^{2\lfloor A \rfloor}),
\end{align}
where the $\alpha_{m,h}$ are linear combinations of the values $\{h(-i(k+\tfrac 12))\}_{k\geq 0}$. Finally, the derivatives $(H^+)^{(j)}$ are bounded on the interval $(0,4\pi)$ for $j\in \{1,2\}$. 
\end{lemma}

\begin{proof}
Let $x\in (0,4\pi)$. Now, the bound~\cite[The formula below (3.1)]{LQ} (see also \cite[3.3 (6)]{Wa}) states that, for $\Re(v)\geq0$,
\begin{equation}
 J_v (x) \ll \frac{x^{\Re(v)}}{|\Gamma(v+\frac 12)|}.
 \label{equation bound bessel}
\end{equation}
 Moreover, Stirling's formula implies that, for $\Re(z) \geq \frac 12$,
\begin{equation}
\frac 1{\Gamma(z)} \ll e^{\frac \pi 2 |\Im(z)| +\Re(z)}|z|^{\frac 12-\Re(z)}.
\label{equation Stirling}
\end{equation}
It follows that the Bessel function 
$$J_{2is}(x) =  \sum_{m=0}^{\infty} \frac{(-1)^m}{m!\Gamma(m+2is+1)} \Big( \frac x2\Big)^{2m+2is} $$
is an absolutely convergent sum (uniformly on compact subsets of $ \mathbb C$) of entire functions of $s$, and is therefore entire (in $s$). (It is crucial here that $x>0$.)

Now, the bounds~\eqref{equation bound bessel} and~\eqref{equation Stirling} imply that in the region $ \Im(s)\leq 0 $ 
we have the bound  

\begin{equation}
 J_{2is}(x) \ll  x^{-2\Im(s)} e^{\pi| \Re(s)| -2\Im(s)}|\tfrac 12+2is|^{2\Im(s)}.
\label{equation bound bessel imaginary}
\end{equation}
Hence \begin{equation*}
H^{+}(x)=\frac{2i}{\pi}\int_{\mathbb R}h(s)J_{2is}(x)\frac{s}{\cosh(\pi s)}\,\d s
\end{equation*}
is well-defined. 
By the decay of $h$ on horizontal lines, we may shift the contour of integration in the definition of $H^{+}(x)$ to the line $\Im(s) = -\lfloor A \rfloor$. Note that the integrand has simple poles at the points $s=-i(k+\frac 12)$ with $k\in \mathbb Z$ with residues
\[
h(-i(k+\tfrac{1}{2})) J_{2k+1}(x) (-i (k+\tfrac{1}{2})) (-1)^k \frac{i}{\pi},
\]
and thus we reach the formula
\begin{equation}
H^+(x) = \frac{4}{\pi}\sum_{k =0}^{ \lfloor A \rfloor -1}(-1)^k (k+\tfrac 12)h(-i(k+\tfrac 12)) J_{2k+1}(x) +\frac{2i}{\pi}\int_{\Im(s) = -\lfloor A \rfloor}h(s)J_{2is}(x)\frac{s}{\cosh(\pi s)}\,\d s. 
\label{equation shifted H+}
\end{equation}
To bound the integral on the right-hand side, we use the fact that
 if
$|s +i(k+\frac 12)|\geq \frac12$ for every $k \in \mathbb Z$, then
$$\frac 1{\cosh(\pi s)} \ll e^{-\pi|\Re(s)|},$$
and get, using also~\eqref{equation bound bessel imaginary}, that the integral is
\begin{align*}
\ll \int_{\Im(s) = -\lfloor A \rfloor} |h(s)| |s|x^{2\lfloor A \rfloor} e^{  2\lfloor A \rfloor}|\tfrac 12+2is|^{-2\lfloor A \rfloor}  |\d s| \ll_A  x^{2\lfloor A \rfloor}, 
\end{align*} 
which proves~\eqref{equation analytic continuation H+}. The second expansion~\eqref{eq:TaylorExp} follows from the familiar Taylor expansion of Bessel functions (see e.g.~\cite[2.11~(1)]{Wa}).

As for differentiability of $H^+(x)$, note that the sum on the right-hand side of~\eqref{equation shifted H+} is clearly twice differentiable. As for the integral, it is also twice differentiable by the identities~\cite[3.2~(2)]{Wa}
$$ 2J'_{2is}(x) = J_{i(2s +i)}(x)  - J_{i(2s-i)}(x) ; \qquad  4J''_{2is}(x) = J_{2i(s +i)}(x)  -2J_{2is}(x)  + J_{2i(s-i)}(x) $$
as well as the bound~\eqref{equation bound bessel imaginary} (recall that $A>13$). The same argument also shows that the first two derivatives of $H^+(x)$ are bounded on $(0,4\pi)$.
\end{proof}

Next, we compute the asymptotic size of the total mass $\Omega^*(h,N)$. Thanks to Lemmas~\ref{corollary kim sarnak},~\ref{Kuznetsov1}, and~\ref{lemma bound H+ small x}, we have that
\begin{align}\label{Total mass}
\Omega^*(h,N)&=\sum_{f\in B^*(N)}\omega_f(N) h(t_f) = \Delta_N^\ast (1, 1; h) =\frac1{\pi^2}\int_{-\infty}^{\infty}h(t)\tanh(\pi t)t\,\d t+O_{A,h}(N^{-1}).
\end{align} 
In other words, the total weight is asymptotically constant (depending on $h$).

It is now time to combine the work in this section to write $\mathcal{D}^\ast(\phi, h; N)$ in terms of averages of Kloosterman sums.
\begin{lemma}
\label{lemma apply Kuznetsov}
Let $h\in \mathscr{H}(A,\delta)$ with $A>13$, let $N$ be a prime and let  $\phi$ be an even Schwartz function for which  ${\rm supp}(\widehat \phi) \subset (-2,2)$. Then we have the estimate
\begin{multline}\label{multline:Dstarformula}
\mathcal D^*(\phi,h;N) = \widehat{\phi}(0) +\frac{\phi(0)}{2} \\-\frac2{\Omega^*(h,N)} \sum_{ c\equiv 0 \bmod N}c^{-1}\sum_{\substack{ p,\nu \\ p \nmid N}} \frac{ S(p^\nu,1;c)}{p^{\frac \nu2}} \widehat \phi \Big( \frac{\nu\log p}{\log N} \Big)H^{+}\Big(\frac{4\pi p^{\frac \nu 2}}{c}\Big)\frac{\log p}{\log N}  +O\Big( \frac 1{\log N}\Big) .
\end{multline}
\end{lemma}

\begin{proof}
Recalling~\eqref{weighted 1 level density} as well as the fact that for $p=N$ we have the identities $\lambda_f(p^{\nu})=\lambda_f(p)^{\nu}$ and $\lambda_f(p)=\pm p^{-\frac12}$ (see, e.g., \cite[Lemma 4.11]{St}), Lemma \ref{Lemma Explicit Formula} implies the estimate
$$\mathcal D^*(\phi,h;N)= \widehat{\phi}(0) +\frac{\phi(0)}2   -\frac2{\Omega^*(h,N)}\sum_{f \in B^*(N)} \omega_f(N) h(t_f)\sum_{\substack{p\nmid N\\ \nu\geq1}} \frac{\lambda_f(p^{\nu}) }{p^{\frac \nu2}} \widehat \phi \Big( \frac{\nu\log p}{\log N} \Big)\frac{\log p}{\log N}   
 +O\Big(\frac 1{\log N}\Big).$$
 The next step is to show that the corresponding estimate where the sum on the right-hand side runs over all $f \in B(N)$, that is
\begin{multline}
\mathcal D^*(\phi,h;N)= \widehat{\phi}(0) +\frac{\phi(0)}2 
 -\frac2{\Omega^*(h,N)}\sum_{f \in B(N)} \omega_f(N) h(t_f)\sum_{\substack{p\nmid N\\ \nu\geq1}} \frac{\lambda_f(p^{\nu}) }{p^{\frac \nu2}} \widehat \phi \Big( \frac{\nu\log p}{\log N} \Big)\frac{\log p}{\log N}   
 +O\Big(\frac 1{\log N}\Big),
 \label{equation this is actually}
\end{multline}
holds.
Indeed, by the definitions of $\Delta_N$ and $\Delta_N^\ast$ and Lemma~\ref{corollary kim sarnak}, the difference between the two expressions is equal to 
\begin{align*}
  &\frac2{\Omega^*(h,N)}\sum_{\substack{p\nmid N\\ \nu\geq1}} \frac{\Delta^*_N(p^\nu,1;h)-\Delta_N(p^\nu,1;h) }{p^{\frac \nu2}} \widehat \phi \Big( \frac{\nu\log p}{\log N} \Big)\frac{\log p}{\log N} \\
&= -\frac4{(N+1) \Omega^*(h,N)}\sum_{\substack{p\nmid N\\ \nu\geq1}} \frac{\Delta_1(p^\nu,1;h) }{p^{\frac \nu2} } \widehat \phi \Big( \frac{\nu\log p}{\log N} \Big)\frac{\log p}{\log N} \\
  &= -\frac4{(N+1)\Omega^*(h,N)}\sum_{f \in B(1)}\omega_f(1) h(t_f)\sum_{\substack{p\nmid N\\ \nu\geq1}} \frac{\lambda_f(p^{\nu}) }{p^{\frac \nu2}} \widehat \phi \Big( \frac{\nu\log p}{\log N} \Big)\frac{\log p}{\log N},
  \end{align*}
which by the explicit formula~\eqref{equation lemma explicit formula 2} with $X=N$ and the Kim-Sarnak bound~\eqref{equation Kim-Sarnak} applied to the term $ p = N$ is equal to
\begin{equation}
  \label{eq:difference}
  \frac2{(N+1)\Omega^*(h,N)}\sum_{f \in B(1)}\omega_f(1) h(t_f)\bigg(\sum_{\gamma_f}\phi\Big(\gamma_f\frac{\log N}{2\pi}\Big)+O\Big(1+\frac{\log(|t_f|+2)}{\log N} \Big) \bigg).
\end{equation}
Writing $\sigma:=\sup ({\rm supp}(\widehat \phi))$, we have the bound
$$ \phi(z) = \int_{\mathbb R} e^{2\pi i z x} \widehat \phi(x) \d x  \ll \int_{\mathbb R} e^{2\pi x| \Im(z) | } |\widehat \phi(x)| \d x \ll  e^{2\pi \sigma |\Im(z)| }. $$
Moreover, integration by parts gives 
$$ \phi(z) = \frac 1{(2\pi i z)^2} \int_{\mathbb R} e^{2\pi i z x} \widehat \phi''(x) \d x  \ll \frac 1{|z|^2}\int_{\mathbb R} e^{2\pi |x \Im(z) | } |\widehat \phi''(x)| \d x \ll \frac 1{|z|^2} e^{2\pi \sigma |\Im(z)| }, $$
and hence we conclude from the above bounds that
$$  \phi(z)  \ll \frac {e^{2\pi \sigma |\Im(z)| }}{1+|z|^2} .$$
As a result, the inner sum in~\eqref{eq:difference} can be bounded as
\[
  \sum_{\gamma_f}\phi\left(\gamma_f \frac{\log N}{2\pi}\right) \ll \sum_{\gamma_f} \frac{N^{\sigma |\Im(\gamma_f)|}}{ 1+ |\gamma_f|^2}  \ll N^{\frac{\sigma}{2}}\sum_{\gamma_f} \frac 1{ 1+ |\gamma_f|^2}, 
  \]
which  by~\cite[Theorem 5.38]{IK} is $\ll N^{\frac \sigma 2}\log(N(|t_f|+2)) $. Hence,~\eqref{eq:difference} yields an acceptable error term.
  
Coming back to the expression~\eqref{equation this is actually} for the one-level density, we may apply Lemma~\ref{Kuznetsov1} and see that the third term in this expression equals
\begin{multline}
\frac{2\big(1 + \frac1N\big)}{\Omega^*(h,N)\pi N}  \sum_{\substack{p\nmid N \\ \nu \geq 1}} \frac{1}{p^{\frac \nu2}} \widehat \phi \Big( \frac{\nu\log p}{\log N} \Big)\frac{\log p}{\log N}\int_{-\infty}^{\infty}h(t)\frac{\sigma_{2it}(p^\nu)}{p^{it\nu}|\zeta_N(1+2it)|^2}\,\d t\\
-\frac2{\Omega^*(h,N)} \sum_{ c\equiv 0 \bmod N}c^{-1}\sum_{\substack{p\nmid N \\ \nu \geq 1}} \frac{S(p^\nu,1;c) }{p^{\frac \nu2}} \widehat \phi \Big( \frac{\nu\log p}{\log N} \Big)H^{+}\Big(\frac{4\pi p^{\frac \nu 2}}{c}\Big)\frac{\log p}{\log N}.
\label{equation bound on zeta}
\end{multline}
Recalling~\eqref{Total mass} we see that the first term in~\eqref{equation bound on zeta} is $\ll_\eps N^{\frac{\sigma}2+\eps-1}$, and the proof is completed.
\end{proof}

We are now ready to express the one-level density as the expected Katz--Sarnak main term plus a sum over primes, which we will bound in the next section.

\begin{proposition} \label{Prop truncation} 
Let $h\in \mathscr{H}(A,\delta)$ with $A> 13 $, let $N$ be a prime and let  $\phi$ be an even Schwartz function for which  ${\rm supp}(\widehat \phi) \subset (-2, 2)$. Then we have the estimate
\begin{multline*}
\mathcal D^*(\phi,h;N) = \widehat{\phi}(0) +\frac{\phi(0)}{2} \\ +O \Bigg(    \sum_{ \substack{ c\equiv 0 \bmod N  \\ c< N^2}}\frac 1{c\varphi(c)} \sum_{ \substack{ d\mid c \\ d\neq 1}} d  \underset{\chi \bmod d}{{\sum}^*}  \Bigg| \sum_{n} \frac{\chi(n)}{n^{\frac 12}} \widehat \phi \Big( \frac{\log n}{\log N} \Big)H^{+}\Big(\frac{4\pi n^{\frac 1 2}}{c}\Big)\frac{\Lambda(n)}{\log N} \Bigg| +\frac 1{\log N}  \Bigg).
\end{multline*}
\end{proposition}
\begin{proof}
Starting from Lemma~\ref{lemma apply Kuznetsov}, we may apply the same argument as that in~\cite[Proof of Lemma 3.1]{DFS2}. This way we see that the third term on the right-hand side of~\eqref{multline:Dstarformula} equals
\begin{align*}
   &  -\frac2{\Omega^*(h,N)} \sum_{ \substack{ c\equiv 0 \bmod N  \\ c< N^2}}\frac 1{c\varphi(c)} \sum_{\substack{\chi \bmod c \\ \chi \neq \chi_0}} \tau( \overline \chi )^2 \sum_{\substack{p \nmid c \\ \nu\geq 1}} \frac{\chi(p^\nu) }{p^{\frac \nu2}} \widehat \phi \Big( \frac{\nu\log p}{\log N} \Big)H^{+}\Big(\frac{4\pi p^{\frac \nu 2}}{c}\Big)\frac{\log p}{\log N} +O\Big( \frac 1{\log N}\Big),
\end{align*}
where $\tau(\chi)$ is the Gauss sum associated to $\chi$.
The claimed estimate follows from decomposing the sum over characters into primitive characters modulo their respective conductors, and bounding trivially the contribution of the primes dividing~$c$.
\end{proof}

To bound the error term in Proposition \ref{Prop truncation}, we will use Mellin inversion and hence we will need a bound on the Mellin transform
\begin{equation}
\label{equation definition Psi}
\Psi_{\phi,N,c,h}(s) := \frac{1}{\log N} \int_0^{\infty} x^{s-1}\widehat \phi \Big( \frac{\log x}{\log N} \Big)  H^+(4\pi \sqrt{x}/c) \d x. 
\end{equation}

\begin{lemma}\label{lemma:bound Psi}
Let $\sigma \in (0, 2)$ be fixed. Let $N$ be a prime, let $c \geq N$, let $\phi$ be an even Schwartz function for which ${\rm supp}(\widehat \phi) \subset (-\sigma, \sigma)$, and let $h\in \mathscr{H}(A,\delta)$. Then, the function $\Psi_{\phi,N,c,h}(s)$ defined in~\eqref{equation definition Psi} is entire and satisfies the bound
	\begin{align*}
		\Psi_{\phi,N,c,h}(s) \ll \frac{N^{\sigma \lvert\Re(s) + \frac{1}{2} \rvert}}{(\lvert s \rvert +1)^2 c}.
	\end{align*}
\end{lemma}

\begin{proof}
The fact that $\widehat \phi$ is supported in $(-2,2)$ immediately implies that $\Psi_{\phi,N,c,h}(s)$ is entire.
Consider first the case $\lvert s \rvert >1$.
	We change variables $u= \frac{\log x}{\log N}$ in the definition of $\Psi_{\phi,N,c,h}(s)$ and then integrate by parts twice:
	\begin{align*}
		\Psi_{\phi,N,c,h}(s) &= \int_{\mathbb{R}} N^{us} \widehat\phi(u) H^+(4 \pi N^{\frac{u}{2}}/c) \d u \\
		&= -\int_{\mathbb{R}} \frac{N^{us}}{s \log N} \Big( \widehat\phi'(u) H^+(4 \pi N^{\frac{u}{2}}/c) + \widehat\phi(u) \frac{2 \pi}{c} N^{\frac{u}{2}}(\log N) (H^+)'(4 \pi N^{\frac{u}{2}}/c)\Big) \d u \\
		&= \int_{\mathbb{R}} \frac{N^{us}}{(s \log N)^2} \Big( \widehat\phi''(u) H^+(4 \pi N^{\frac{u}{2}}/c)   + \widehat\phi'(u) \frac{4 \pi}{c} N^{\frac{u}{2}}(\log N) (H^+)'(4 \pi N^{\frac{u}{2}}/c) \\& 
		\hspace{1cm} +   \widehat\phi(u) \frac{\pi}{c} N^{\frac{u}{2}}(\log N)^2 (H^+)'(4 \pi N^{\frac{u}{2}}/c) + \widehat\phi(u) \frac{4 \pi^2}{c^2} N^{u}(\log N)^2 (H^+)''(4 \pi N^{\frac{u}{2}}/c)\Big) \d u.
	\end{align*}
By Lemma~\ref{lemma bound H+ small x}, we deduce that
\begin{align*}
	\Psi_{\phi,N,c,h}(s) &\ll_A \int_{-\sigma}^{\sigma} \frac{N^{u \Re(s)}}{\lvert s\rvert^2}  \Big( \frac{N^{\frac{u}{2}}}{c}\Big)  \d u 
	\ll_A \frac{N^{\sigma \lvert\Re(s) + \frac{1}{2} \rvert}}{\lvert s \rvert^2 c} .
\end{align*}
Finally, for $\lvert s \rvert\leq1$ the bound $\Psi_{\phi,N,c,h}(s)  \ll \frac{N^{\sigma \lvert\Re(s) + \frac{1}{2} \rvert}}{c}$ follows directly from the definition and Lemma~\ref{lemma bound H+ small x}.
\end{proof}

We end this section by giving a brief sketch of the proof of Theorem~\ref{theorem second}, which is very similar to that of~\cite[Theorem 1.1 \& Remark 4.2]{DFS2}.
\begin{proof}[Proof of Theorem~\ref{theorem second}]
The first step is to rewrite the sum over primes in Proposition~\ref{Prop truncation} in terms of a contour integral. By Mellin inversion and the definition of $\Psi_{\phi,N,c,h}$ in~\eqref{equation definition Psi}, we obtain
\begin{multline*}
		\mathcal D^*(\phi,h;N)  = \widehat{\phi}(0) + \frac{\phi(0)}{2} 	+ O\Big(\frac{1}{\log X}\Big)
	\\  + O\bigg(   \sum_{\substack{c \equiv 0 \bmod N \\ c < N^{2}}}  \frac{1}{c\varphi(c)} \sum_{\substack{d\mid c \\d \neq 1}} d \underset{\chi \bmod d}{{\sum}^*}   \bigg|
  \int_{(2)}  \frac{L'(s+\frac12,\chi)}{L(s+\frac12,\chi)}\Psi_{\phi,N,c,h}(s) \d s \bigg|    \bigg).
	\end{multline*}
Then, using Lemma~\ref{lemma:bound Psi} in place of~\cite[Lemma 3.3]{DFS2}, we move the contour integral to the left as in~\cite[Proof of Proposition 3.4]{DFS2} in the case $k=2$, and obtain the estimate
\begin{multline}\label{equation dirichlet 3}
	\mathcal D^*(\phi,h;N) = \widehat{\phi}(0) + \frac{\phi(0)}{2}  	+ O\Big(\frac{1}{\log X}\Big)
	 \\ + O_{\eps}\bigg( N^{\eps} \sup_{ \frac N2 \leq D <  N^{2}} \sup_{\frac12\leq\beta<1} \sup_{1\leq T \leq N^5 }  \frac{N^{\sigma\beta}}{D^{2} T^2} \sum_{ \substack{d \sim D \\ d\equiv 0 \bmod N }} \underset{\chi \bmod d}{{\sum}^*} 
	\sum_{\substack{\rho_\chi \\ \beta\leq\Re(\rho_\chi) < \beta+ \frac{1}{\log N} \\ T-1\leq \lvert\Im(\rho_\chi)\rvert<2T } } 1\bigg),
\end{multline}
where $\rho_{\chi}$ is running over the non-trivial zeros of  the Dirichlet $L$-function~$L(s,\chi)$ and $d\sim D$ means $D<d \leq 2D$.
The estimate~\eqref{equation dirichlet 3} is actually identical to~\cite[Proposition 3.4]{DFS2} with $k = 2$, so the rest of the argument is the same as in~\cite[Proof of Theorem 1.1 \& Remark 4.2]{DFS2}. 
\end{proof}

\section{Handling the prime sum using Dirichlet polynomials}\label{sec:Primesum}

The final step in the proof of Theorem \ref{theorem main} is to bound the prime sum in Proposition \ref{Prop truncation}. In particular, we will prove the following proposition which together with Proposition~\ref{Prop truncation} immediately implies Theorem~\ref{theorem main}.

\begin{proposition}\label{prop 1}
	Let $\varepsilon > 0$ be fixed. Let $h\in \mathscr{H}(A,\delta)$ with $A>13 $, let $N$ be prime and let  $\phi$ be an even Schwartz function for which  ${\rm supp}(\widehat \phi) \subset (- \frac{15}8+\eps, \frac{15}8-\eps)$.
	Then
	\begin{align*}
		\sum_{ \substack{ c\equiv 0 \bmod N  \\ c< N^2}}\frac 1{c\varphi(c)} \sum_{\substack{d \mid c \\ d \neq 1}} d \sumstar  \Bigg| \sum_{n} \frac{\chi(n) }{n^{\frac{1}{2}}} \widehat \phi \Big( \frac{\log n}{\log N} \Big)H^{+}\Big(\frac{4\pi n^{\frac{1}{2}}}{c}\Big)\frac{\Lambda(n)}{\log N} \Bigg|  \ll N^{-\frac \eps{10}}.
	\end{align*}
\end{proposition}

Let us first show that it suffices to prove that, for any $c \equiv 0 \pmod{N}$ with $c < N^2$ and any $d \mid c$ with $d \neq 1$, we have
\begin{equation}
\label{eq:dsumclaim}
\sumstar  \Bigg| \sum_{n} \frac{\chi(n) }{n^{\frac{1}{2}}} \widehat \phi \Big( \frac{\log n}{\log N} \Big)H^{+}\Big(\frac{4\pi n^{\frac{1}{2}}}{c}\Big)\frac{\Lambda(n)}{\log N} \Bigg| \ll \frac{c}{d} N^{1-\frac \eps 5}.
\end{equation}
Indeed, assuming~\eqref{eq:dsumclaim}, we obtain
\begin{multline*}
\sum_{ \substack{ c\equiv 0 \bmod N  \\ c< N^2}}\frac 1{c\varphi(c)} \sum_{\substack{d \mid c \\ d \neq 1}} d \sumstar \Bigg| \sum_{n} \frac{\chi(n) }{n^{\frac{1}{2}}} \widehat \phi \Big( \frac{\log n}{\log N} \Big)H^{+}\Big(\frac{4\pi n^{\frac{1}{2}}}{c}\Big)\frac{\Lambda(n)}{\log N} \Bigg| \\
\ll N^{1-\frac \eps 5} \sum_{ \substack{ c\equiv 0 \bmod N  \\ c< N^2}}\frac {1}{\varphi(c)} \sum_{\substack{d \mid c \\ d \neq 1}} 1 \leq N^{1-\frac \eps 5} \sum_{c' < N} \frac{\tau(c'N)}{\varphi(c'N)} \ll N^{-\frac \eps {10}}.
\end{multline*}

Our first step in proving~\eqref{eq:dsumclaim} is an application of Heath-Brown's identity (see for example \cite[Proposition 13.3]{IK}  with $z = N^{ \frac 1{10}}$ and $K = 20$; since the sum over $n$ in~\eqref{eq:dsumclaim} is supported on $n \leq N^{\frac {15}8-\eps}$, Heath-Brown's identity is applicable with these choices). We obtain that the left-hand side of~\eqref{eq:dsumclaim} equals
\begin{align*}
&\sumstar \Bigg|\sum_{j=1}^{20}(-1)^j\binom{20}{j}\sum_{\substack{n_1, \dotsc, n_{2j} \\
	n_{j+1},\dots,n_{2j} \leq N^{ \frac 1{10}}}}\frac{\log(n_j)\mu(n_{j+1})\dotsm \mu(n_{2j})\chi(n_1\dotsm n_{2j})}{(n_1\dotsm n_{2j})^{ \frac 12}} \\
	& \qquad \qquad \qquad \times \widehat \phi \Big( \frac{\log (n_1\dotsm n_{2j})}{\log N} \Big)H^{+}\Big(\frac{4\pi (n_1\dotsm n_{2j})^{\frac{1}{2}}}{c}\Big)\frac{1}{\log N} \Bigg|.
\end{align*}
We split all the variables $n_\ell$ into dyadic ranges $n_\ell \sim N_\ell$ (i.e. $n_\ell \in (N_\ell, 2N_\ell]$) with $N_\ell \geq \frac 12$. To unify the notation, in the case $j < 20$, we add dumb variables $n_{2j+1}, \dotsc, n_{40}$ with $N_{2j+1} = \dotsb = N_{40} = \frac 12$. Now~\eqref{eq:dsumclaim} follows once we have shown that
\begin{align}
\begin{aligned}
\label{eq:afterH-B}
&\sumstar \Bigg|\sum_{\substack{n_1, \dotsc, n_{40} \\ n_j \sim N_j}} \frac{a_1(n_1) \dotsm a_{40}(n_{40}) \chi(n_1 \dotsm n_{40})}{(n_1\dotsm n_{40})^{\frac 12}} \\
	& \qquad \qquad \times \widehat \phi \Big( \frac{\log (n_1\dotsm n_{40})}{\log N} \Big)H^{+}\Big(\frac{4\pi (n_1\dotsm n_{40})^{\frac{1}{2}}}{c}\Big)\frac{1}{\log N} \Bigg| \ll \frac{c}{d}N^{1-\frac \eps 4}
\end{aligned}
\end{align}
whenever, for each $j = 1, \dotsc, 40$, we have that $N_j \geq \frac 12$ and that $a_j(n)$ is $1$, $\mu(n)$ or $\log(n)$, and $a_j(n)$ is 1 or $\log(n)$ if $N_j>N^{\frac 1{ 10}}.$ 

Recalling the support of $\widehat{\phi}$, we see that it suffices to consider the case
\[
N_1 \dotsm N_{40} \leq N^{\frac {15}8-\varepsilon}.
\]
Applying Mellin inversion, the left-hand side of~\eqref{eq:afterH-B} equals
\begin{align*}
\sumstar
  \Bigg| \frac1{2\pi} \int_{-\infty}^{\infty} \sum_{\substack{n_1, \dotsc, n_{40} \\ n_j \sim N_j}} \frac{a_1(n_1) \dotsm a_{40}(n_{40}) \chi(n_1 \dotsm n_{40})}{(n_1\dotsm n_{40})^{ \frac 12+it}} \Psi_{\phi,N,c,h}(it) \d t\Bigg|,
\end{align*}
where $\Psi_{\phi,N,c,h}$ is as in~\eqref{equation definition Psi}. By Lemma \ref{lemma:bound Psi} and the fact that $c \in [\max\{N, d\}, N^2]$ we see that
\[
  \Psi_{\phi, N, c, h}(it) \ll \frac{N^{\frac{15}{16}-\frac{\varepsilon}{2}}}{(t^2+1)c} \ll \frac{c N^{\frac{15}{16}-\frac\varepsilon2}}{(t^2+1) (N+d)^2}.
  \]
  This estimate may look lossy, but the most difficult case to handle is $c = d = N$ in which case it is sharp. Now~\eqref{eq:afterH-B}, and thus Proposition~\ref{prop 1}, follows once we have shown the following proposition.
  
\begin{proposition}\label{prop:charsumclaim}
Let $d \geq 2$ be an integer, let $N \geq 2$ and $N_1, \dotsc, N_{40} \geq \frac 12$. For each $j = 1, \dotsc, 40$, let $a_j(n)$ be  $1$, $\mu(n)$ or $\log(n)$, and assume that $a_j(n)$ is $1$ or $\log(n)$ for every $j$ for which $N_j>N^{\frac 1{10}}$. Assume also that
\begin{equation}
\label{eq:Niproductbound}
N_1 \dotsm N_{40} \leq N^{ \frac{15}8-\varepsilon}.
\end{equation}
Then
\begin{equation}
\label{eq:charsumclaim}
\int_{-\infty}^\infty\sumstar\left|\prod_{j\leq 40}\sum_{n_j\sim N_j}\frac{a_j(n_j)\chi(n_j)}{ n_{j}^{ \frac 12+it}}\right|\frac{\d t}{t^2+1}\ll \frac{(N+d)^2N^{ \frac 1{16}}}{d} (\log(dN))^{O(1)}.
\end{equation}
\end{proposition}

In the proof of Proposition~\ref{prop:charsumclaim}, we use the following large sieve type lemmas. The proofs are standard, but we include them for completeness.
	\begin{lemma}\label{lemma orthogonality}
		Let $d \in \mathbb{N}$ and $X \geq 1$. For any complex numbers $a(n)$, we have
		$$\sumstar\bigg|\sum_{n\leq X}\chi(n)a(n)\bigg|^2\ll (d+X)\sum_{n\leq X}|a(n)|^2.$$
	\end{lemma}
\begin{proof}
This follows easily from the orthogonality of characters. Indeed, using also the inequality $|a(n_1) a(n_2)| \leq (|a(n_1)|^2+|a(n_2)|^2)/2$ and symmetry, we obtain
\begin{align*}
\sumstar\bigg|\sum_{n\leq X}\chi(n)a(n)\bigg|^2 &\leq \varphi(d)\sum_{\substack{n_1, n_2 \leq X \\ n_1 \equiv n_2 \bmod{d}}} |a(n_1) a(n_2)| \leq \varphi(d)\sum_{\substack{n_1, n_2 \leq X \\ n_2 \equiv n_1 \bmod{d}}} |a(n_1)|^2 \\
&\leq \varphi(d) \left(\frac{X}{d}+1\right) \sum_{n \leq X} |a(n)|^2,
\end{align*}
which gives the desired estimate.
\end{proof}
	\begin{lemma}\label{lemma fourth moment}
Let $d \geq 2$ be an integer and let $X \geq 2$.
		\begin{enumerate}
			\item[(i)] We have, for any bounded complex coefficients  $b(n)$, 
			\begin{equation*}
				\int_{-\infty}^{\infty}\sumstar\Bigg|\sum_{n\leq X}\frac{b(n)\chi(n)}{n^{ \frac 12+it}}\Bigg|^4\frac{\d t}{t^2+1}\ll (d+X^2) (\log X)^4.
			\end{equation*}
		\item[(ii)] Let either $a(n)=\log n$ for each $n \leq X$, or $a(n)=1$ for each $n \leq X$. Then
		\begin{equation*}
			\int_{-\infty}^{\infty}\sumstar\Bigg|\sum_{n\leq X}\frac{a(n)\chi(n)}{n^{\frac 12+it}}\Bigg|^4\frac{\d t}{t^2+1}\ll d (\log(dX))^{13}.
		\end{equation*}
		\end{enumerate}
\end{lemma} 

\begin{proof} Applying Lemma~\ref{lemma orthogonality} with 
\[
a(n) = \sum_{\substack{n = n_1 n_2 \\ n_1, n_2 \leq X}} \frac{b(n_1) \overline{b(n_2)}}{n^{ \frac 12+it}} \ll \mathbf{1}_{n \leq X^2} \frac{\tau(n)}{n^{ \frac 12}},
\]
we obtain, for any $t \in \mathbb{R}$,
\[
\sumstar\Bigg|\sum_{n\leq X}\frac{b(n)\chi(n)}{n^{\frac 12+it}}\Bigg|^4 \ll 	(d+X^2)\sum_{n \leq X^2} \frac{\tau(n)^2}{n} \ll (d+X^2) \prod_{p \leq X^2} \left(1+\frac{4}{p}\right) \ll (d+X^2) (\log X)^4,
\]
and part (i) follows.

We now move to part (ii). Partial summation and H\"older's inequality give, for any $t \in \mathbb{R}$,
$$
\sumstar\Bigg|\sum_{n\leq X}\frac{\chi(n)\log n}{n^{ \frac 12+it}}\Bigg|^4\ll(\log X)^4\sumstar\Bigg|\sum_{n\leq X}\frac{\chi(n)}{n^{ \frac 12+it}}\Bigg|^4+(\log X)^3\sumstar\int_1^X\Bigg|\sum_{n\leq u}\frac{\chi(n)}{n^{\frac 12+it}}\Bigg|^4\frac{\d u}{u}.
$$
Hence it suffices to show that, for any $X \geq 2$,
		\begin{equation*}
			\int_{-\infty}^{\infty}\sumstar\Bigg|\sum_{n\leq X}\frac{\chi(n)}{n^{\frac 12+it}}\Bigg|^4\frac{\d t}{t^2+1}\ll d (\log(dX))^{9}.
		\end{equation*}

First note that for $|t| \geq X^2$, we have similarly to part (i)
\begin{align*}
\int_{|t|\geq X^2}\sumstar\Bigg|\sum_{n\leq X}\frac{\chi(n)}{n^{\frac 12+it}}\Bigg|^4\frac{\d t}{t^2+1}&\ll \int_{|t|\geq X^2} (d+X^2) \sum_{n \leq X^2} \frac{\tau(n)^2}{n} \frac{\d t}{t^2+1} \\
&\ll \frac{d+X^2}{X^2} (\log X)^4,
\end{align*}
which is acceptable. Hence we can restrict the integral to $|t| < X^2$. Next, by Perron's formula {(see for example \cite[Theorem 7.2]{Ko})}, we have, for any $V \geq 2$,
\begin{equation*}
	\begin{aligned}
	&\int_{-X^2}^{X^2}\sumstar\Bigg|\sum_{n\leq X}\frac{\chi(n)}{n^{ \frac 12+it}}\Bigg|^4\frac{\d t}{t^2+1}\\
	&\ll\int_{-X^2}^{X^2}\sumstar\Bigg|\int_{1-iV}^{1+iV}L(s+\tfrac 12+it,\chi)\frac{X^s}{s}\d s\Bigg|^4\frac{\d t}{t^2+1}+\frac{dX^4(\log X)^4}{V^4}+d.
\end{aligned}
\end{equation*}
Choosing $V = (dX)^2$, the error terms are certainly acceptable, and it suffices to show that, for any $t \in [-X^2, X^2]$, we have 
\begin{equation}
\label{eq:4thclaimnotint}
\sumstar\left|\int_{1-i(dX)^2}^{1+i(dX)^2}L(s+\tfrac 12+it,\chi)\frac{X^s}{s}\d s\right|^4 \ll d (1+|t|) (\log(dX))^8.
\end{equation}
Next, we shift the integral to $\Re(s)=\alpha:=1/\log (dX)$ (note that since $d \geq 2$ and we sum over primitive characters, we do not cross any poles, and the contribution of the horizontal lines is acceptable by the convexity bound) and use H\"older's inequality, which yields
\begin{equation}\label{fourth moment holder and perron}
	\begin{aligned}
	&\sumstar\left|\int_{1-i(dX)^2}^{1+i(dX)^2} L(s+\tfrac 12+it,\chi)\frac{X^s}{s}\d s\right|^4\\
	&\ll\sumstar\int_{-(dX)^2}^{(dX)^2} \left| L\left( \tfrac12+\alpha+it+iu,\chi\right)\right|^4\frac{\d u}{|u|+\alpha}\left(\int_{-(dX)^2}^{(dX)^2}\frac{\d u}{|u|+\alpha}\right)^3+\frac1 {X^2 d^5} \\
	&\ll (\log (dX))^3 \sumstar\int_{-(dX)^2+t}^{(dX)^2+t} \left| L\left( \tfrac12+\alpha+iu,\chi\right)\right|^4\frac{\d u}{|u-t|+\alpha}+\frac1 {X^2 d^5}.
\end{aligned}
\end{equation}

To bound the last expression, we split the integral into dyadic ranges according to the size of $|u-t|$ and use an estimate for the fourth moment of Dirichlet $L$-functions (see for example \cite[Theorem 10.1]{Mo2}). This yields that the first term on the right-hand side of~\eqref{fourth moment holder and perron} is
\begin{equation*}
	\begin{aligned}
	&\ll (\log(dX))^4\sumstar \hspace{.2cm}\int\limits_{|u-t|\leq 2}\left| L\left(\tfrac12+\alpha+iu,\chi\right)\right|^4\d u\\
	&\hspace{1cm}+(\log(dX))^4 \max_{2 \leq T_0 \leq 2(dX)^2} \frac{1}{T_0}\sumstar\hspace{.2cm}\int\limits_{T_0\leq|u-t|\leq 2T_0}\left| L\left(\tfrac 12+\alpha+iu,\chi\right)\right|^4\d u\\
	&\ll (\log(dX))^8 \max_{2 \leq T_0 \leq 2(dX)^2} \frac{d(T_0+|t|)}{T_0} \ll d(1+|t|)(\log(dX))^8,
\end{aligned}
\end{equation*}
and so~\eqref{eq:4thclaimnotint} holds. This completes the proof of part (ii).
\end{proof}

Let us now return to proving Proposition~\ref{prop:charsumclaim}. We use the following lemma to split into two cases according to the sizes of $N_j$.
 Here, the idea is that if we can partition $ \{ 1,\dots,40\} $ into two subsets $I, J $ for which $\prod_{j \in I} N_j  ,\prod_{j \in J} N_j$ are of roughly similar size, then the mean value estimates on Dirichlet polynomials of Lemma~\ref{lemma orthogonality} will be sufficient. On the other hand, in the case when such a balanced decomposition is not possible, some $N_j$ must be relatively large (or the product of all $N_j$ must be small) and we can apply the fourth moment bound of Lemma~\ref{lemma fourth moment}.

\begin{lemma}\label{le:splitting}
Let $\varepsilon > 0$ be fixed, let $N \geq 2$ be sufficiently large, and let $N_1, \dotsc, N_{40} \geq \frac 12$ be such that
\begin{equation}
\label{eq:Niproductbound2}
N_1 \dotsm N_{40} \leq N^{\frac{15}8-\varepsilon}.
\end{equation}
Then at least one of the following holds.
\begin{enumerate}
\item[(a)] There exists a set $I\subseteq\{1,\dots,40\}$ such that 
\begin{equation}
\label{eq:Iprod}
N^{ \frac 34+\frac \eps {100}} < \prod_{j\in I}N_j\leq N^{\frac 98-\frac \eps {100}}.
\end{equation}
\item[(b)] There exist two distinct indices $
j_1,j_2 \in \{1, \dotsc, 40\}$ such that 
\begin{equation}
\label{eq:j1j2} 
\prod_{j \in\{1,\dotsc,40\}\setminus{\{j_1,j_2\}}}N_j\leq N^{\frac 98-\frac \eps {100}}.
\end{equation}
\end{enumerate}
\end{lemma}
\begin{proof}Assume for contradiction that neither (a) nor (b) holds. Consider subsets $I' \subseteq \{1,\dots,40\}$ such that $$P:=\prod_{j\in I'}N_j \leq N^{\frac 34+\frac \eps {100}}.$$ Let $I\subseteq\{1,\dots,40\}$ be one of such subsets containing the maximal number of elements, and let $J=\{1,\dots,40\}\setminus I$. 
We must have $|J|\geq 3$, as otherwise (b) holds for any $\{j_1, j_2\} \supseteq J$. Furthermore, for every $j\in J$, we must have $PN_j>  N^{\frac 98-\frac \eps {100}}$, as otherwise either we obtain a contradiction with the maximality of $I$ or (a) holds. 
Thus, for every $j\in J$, we have $N_j> N^{\frac 98-\frac \eps{ 100}}/P$. However, using also the definition of $P$, we have
\[
  \prod_{j = 1}^{40} N_j> P \left(\frac{N^{\frac 98-\frac \eps {100}}}{P}\right)^3 = \frac{N^{\frac{27} 8-\frac {3\eps} {100}}}{P^2} \geq \frac{N^{\frac{27} 8-\frac {3\eps} {100}}}{N^{\frac 32+\frac \eps {50}}} = N^{\frac{15} 8-\frac \eps {20}}
\]
which contradicts~\eqref{eq:Niproductbound2}.
\end{proof}

Now, we are ready to prove Proposition~\ref{prop:charsumclaim}.

\begin{proof}[Proof of Proposition~\ref{prop:charsumclaim}]Notice first that the right-hand side of~\eqref{eq:charsumclaim} satisfies
  \begin{equation}\label{eq:TailorMaass}
  \frac{(N+d)^2N^{ \frac 1{16}}}{d} \gg \frac{(d^2 + d^{\frac 32} N^{\frac 12} + dN) N^{\frac 1{16}}}{d} = dN^{\frac 1{16}} + d^{\frac 12}N^{\frac 9{16}} + N^{\frac {17}{16}}.
\end{equation}
Lemma~\ref{le:splitting} allows us to split into two cases. First, we assume that Lemma~\ref{le:splitting}(a) holds. Let $I$ be as in Lemma~\ref{le:splitting}(a) and let $J = \{1, \dotsc, 40\} \setminus I$.
Note that in this case
\begin{equation}
\label{eq:Jprod}
\prod_{j \in J} N_j \leq \frac{N^{\frac{15} 8-\varepsilon}}{N^{\frac 34+\frac \eps {100}}} \leq N^{\frac 98-\frac \eps {100}}.\end{equation}
Using the Cauchy--Schwarz inequality and Lemma \ref{lemma orthogonality}, we obtain
$$\begin{aligned}
	&\int_{-\infty}^\infty\sumstar\left|\prod_{j\leq 40}\sum_{n_j\sim N_j}\frac{a_j(n_j)\chi(n_j)}{ n_{j}^{\frac 12+it}}\right|\frac{\d t}{t^2+1}\\
	&\ll\int_{-\infty}^{\infty}\left(\sumstar\left|\prod_{j\in I}\sum_{n_j\sim N_j}\frac{a_j(n_j)\chi(n_j)}{n_j^{\frac 12+it}}\right|^2\right)^{\frac 12}\left(\sumstar\left|\prod_{j \in J}\sum_{n_j\sim N_j}\frac{a_j(n_j)\chi(n_j)}{n_j^{\frac 12+it}}\right|^2\right)^{\frac 12}\frac{\d t}{t^2+1}\\
	&\ll \bigg( d+\prod_{j\in I} N_j\bigg)^{\frac 12}\bigg( d+\prod_{j\in J} N_j\bigg)^{\frac 12} (\log(dN))^{O(1)}\\
	&\ll \Bigg( d+d^{\frac 12}\prod_{j\in I}N_j^{\frac 12}+d^{\frac 12}\prod_{j\in J}N_j^{\frac 12}+ \prod_{j =1}^{40} N_j^{\frac 12}\Bigg) (\log(dN))^{O(1)}.
	\end{aligned}
	$$
Now~\eqref{eq:Niproductbound},~\eqref{eq:Iprod}, and~\eqref{eq:Jprod} imply that this is 
\[
\ll (d + d^{\frac 12}N^{\frac 9{16}-\frac \eps {200}}+N^{\frac {15} {16}-\frac \eps 2}) (\log(dN))^{O(1)},
\]
and by~\eqref{eq:TailorMaass} we see that~\eqref{eq:charsumclaim} holds in this case.

Next, we assume that Lemma~\ref{le:splitting}(b) holds and let $j_1,j_2$ be the associated indices. Using H\"older's inequality, we obtain
\begin{align}
\label{eq:Caseb}
\begin{aligned}
	&\int_{-\infty}^{\infty}\sumstar\left|\prod_{j\leq 40}\sum_{n_j\sim N_j}\frac{a_j(n_j)\chi(n_j)}{n_j^{\frac 12+it}}\right|\frac{\d t}{t^2+1}\\
	&\ll\left(\int_{-\infty}^{\infty}\sumstar\left|\prod_{\substack{j \leq 40 \\ j\notin\{j_1,j_2\}}}\sum_{n_j\sim N_j}\frac{a_j(n_j)\chi(n_j)}{n_j^{ \frac 12+it}}\right|^2\frac{\d t}{t^2+1}\right)^{\frac 12}\\
	&\times\prod_{j\in\{j_1,j_2\}}\left(\int_{-\infty}^{\infty}\sumstar\left|\sum_{n_{j}\sim N_j}\frac{a_j(n_j)\chi(n_j)}{n_j^{\frac 12+it}}\right|^4\frac{\d t}{t^2+1}\right)^{\frac 14}.
\end{aligned}
\end{align}
Now, for $\ell = 1, 2$, either $N_{j_\ell} \leq N^{1/10}$ or $a_{j_\ell}(n)$ is $1$ or $\log n_{j_\ell}$. Applying Lemma~\ref{lemma fourth moment} to the fourth moments on the right-hand side of~\eqref{eq:Caseb} and Lemma~\ref{lemma orthogonality} to the second moment and using~\eqref{eq:j1j2}, we obtain that
$$
\begin{aligned}
	&\int_{-\infty}^{\infty}\sumstar\left|\prod_{j=1}^{40}\sum_{n_j\sim N_j}\frac{a_j(n_j)\chi(n_j)}{n_j^{\frac 12+it}}\right|\frac{\d t}{t^2+1}\\
	&\ll\bigg( d+\prod_{j\notin\{j_1,j_2\}} N_j\bigg)^{\frac 12} (d+N^{\frac15})^{\frac 12}(\log(dN))^{O(1)} \\
	&\ll \big(d+d^{\frac 12}N^{\frac 1 {10}}+d^{\frac 12}N^{\frac 9 {16}-\frac \eps {200}} + N^{\frac 9 {16}-\frac \eps {200}+\frac 1 {10}}\big)(\log(dN))^{O(1)}.
\end{aligned}$$
Again we see from~\eqref{eq:TailorMaass} that~\eqref{eq:charsumclaim} holds also in this case.
\end{proof}

\begin{remark}
  The support $(-\frac {15} 8, \frac {15} 8)$ is the limit of this technique. Indeed, extending the support to $(-\frac {15} 8-\varepsilon, \frac {15} 8+\varepsilon)$ would require us to have a variant of Proposition~\ref{prop:charsumclaim} with $N_1 \dotsm N_{40} \leq N^{ \frac{15} 8+\varepsilon}$ and with $N^{\frac 1{16}}$ on the right-hand side of~\eqref{eq:charsumclaim} replaced by $N^{\frac 1{16}-\frac{\varepsilon} 2}$. In particular, we would need to handle the case when $c = d = N,$ $N_1 = \dotsb = N_5 = N^{\frac 38+\frac \eps 5}$, and $N_j = \frac 1 2$ for $j \geq 6$. If we tried to deal with this similarly to case (a) of Lemma~\ref{le:splitting}, we would obtain the bound $d^{\frac 12} N^{ \frac 9{16}+\frac{\varepsilon}{10}}$ which would not be sufficient. The more precise bottleneck seems to be that in this case we cannot get a sufficiently good upper bound for the number of characters modulo $N$ such that
  \begin{equation}\label{eq:worstcase}
\left|    \sum_{n \sim N_1} \frac{\chi(n)}{n^{\frac 12}}\right| \geq N_1^{\frac 16}.
    \end{equation}
By considering the fourth moment of the last expression and ignoring the $t$-aspect in Lemma~\ref{lemma fourth moment}(ii), we can expect that the number of such $\chi$ is at most $N \cdot N_1^{-4 \cdot \frac 16}$, so the contribution of such characters to the left-hand side of~\eqref{eq:charsumclaim} can, as far as we know, be essentially as large as $N N_1^{-4 \cdot \frac 16} N_1^{5 \cdot \frac 16} = N N_1^{\frac 16} \asymp N^{1+\frac 1 {16}+\frac {\varepsilon}{30}}$, which is not acceptable.
  \end{remark}

\begin{remark}\label{remark HB beats zeros}
One might wonder why it is better to use Heath-Brown's identity than zero-density results since these two approaches often lead to the same result as both are based on the same mean and large value results for Dirichlet polynomials. This is the case for instance for obtaining an asymptotic formula for the number of primes in short intervals $(X, X+H]$, previously with $H = X^{\frac 7 {12}+\varepsilon}$ (see e.g.~\cite[Sections 7.1--7.3]{Ha}), and now, thanks to the recent work of Guth and Maynard~\cite{GM} (see in particular their Corollary 1.3 and Remark in the end of Section 13.2), for $H = X^{ \frac{17}{30}+\varepsilon}$.

Let us discuss for instance Huxley's prime number theorem giving primes in intervals of length $H = X^{ \frac 7{12}+\varepsilon}$ and his zero-density estimate $N(\sigma, T) \ll T^{\frac{12}{5}(1-\sigma)} (\log T)^{O(1)}$ (for these, see e.g.~\cite[Chapter 10]{IK}). In the proof of the zero-density estimate (see e.g.~\cite[Chapter 10]{IK}), the first step is to show that there exists a zero-detecting polynomial for which mean value and large value results for Dirichlet polynomials are subsequently applied.  Studying the proof, it turns out that the zero-density estimate can be improved unless the length of the zero-detecting polynomial is very close to $T^{ \frac 25}$ (and $\sigma \approx \frac 34$). In Huxley's case $T \approx X/H \approx X^{\frac 5{12}}$ and so $T^{ \frac 25} \approx X^{ \frac 16}$. Now one encounters this very same worst case scenario when Heath-Brown's identity is applied, in the case when there are six polynomials of length $X^{ \frac 16}$; see~\cite[Section 4]{HB}.

In our case, in Section~\ref{sec:Primesum} we have $N_1 \dotsm N_k < N^{\frac {15} 8-\varepsilon}$, and thus Heath-Brown's identity can lead to a case where all $N_i \approx N^{\frac 25}$ only if $N_1 \dotsm N_k \lesssim N^{\frac 85}$, which is easy to handle. Therefore the worst-case scenario from the zero-density estimate is not occurring, and we can obtain an improvement by applying Heath-Brown's identity directly.

Let us finally point out that the recent breakthrough of Guth and Maynard~\cite{GM} seems not to be directly relevant in our case for two reasons. First, it does not currently provide an improvement in the $\chi$-aspect. Second, our most problematic character sums are so small (see~\eqref{eq:worstcase}) that even the most natural $\chi$-analogue would not be useful --- note that~\cite[Theorem 1.1]{GM} does not improve upon previous results for the number of large values $|\sum_{n \sim N_1} b_n n^{it}| \geq V_1$ when $V_1 \leq N_1^{ \frac 7{10}}$, and~\eqref{eq:worstcase} corresponds to $V_1 = N_1^{ \frac 12+\frac 16} < N_1^{\frac 7{10}}$.
\end{remark}

\section{Proof of Theorem \ref{theorem main holom}} \label{sec:holom}
In this section we use the notation and set-up from Section~\ref{ssec:introholom}. Recall in particular that $X =k^2N$ is the size of the analytic conductor of $L(s,f)$ and $$\Theta_k = 2-\frac{1}{5k-2}$$ for all even integers $k\geq2$. By \cite[Lemma 3.1 and the proof of Lemma 3.2]{DFS2} we have the following variant of Proposition~\ref{Prop truncation}.

\begin{proposition} \label{Prop truncation holo} 
Let $k \geq 2$ be a fixed even integer, 
let $N$ be a prime and let $\phi$ be an even Schwartz function for which  ${\rm supp}(\widehat \phi) \subset (-2, 2)$. Then we have the estimate
\begin{multline*}
\mathcal D_{k, N}^*(\phi;X) = \widehat{\phi}(0) +\frac{\phi(0)}{2} \\ +O_k \Bigg(   \frac{1}{\log X} \sum_{ \substack{ c\equiv 0 \bmod N  \\ c< N^{1+\frac{1}{2k-3}}}}\frac 1{c\varphi(c)} \sum_{ \substack{ d\mid c \\ d\neq 1}} d  \underset{\chi \bmod d}{{\sum}^*}  \Bigg| \sum_{n} \frac{\Lambda(n) \chi(n)}{n^{\frac 12}} \widehat \phi \Big( \frac{\log n}{\log X} \Big)J_{k-1}\Big(\frac{4\pi n^{\frac 1 2}}{c}\Big) \Bigg| +\frac 1{\log X}  \Bigg).
\end{multline*}
\end{proposition}

This time we need a variant of Lemma~\ref{lemma:bound Psi} which gives a bound for the relevant Mellin transform
\begin{equation}
\label{equation definition Psi holom}
\Psi^\flat_{\phi,X,c,k}(s) := \frac{1}{\log X} \int_0^{\infty} x^{s-1}\widehat \phi \Big( \frac{\log x}{\log X} \Big)  J_{k-1}(4\pi \sqrt{x}/c) \d x.
\end{equation}
Such a bound is provided by~\cite[Lemma 3.3]{DFS2} which we now state.

\begin{lemma}\label{lemma:bound Psi holom2}
Let $\sigma \in (0, 2)$ be fixed. Let $k \geq 2$ be a fixed even integer, let $N$ be a prime, and let $c \geq N$. Let $\phi$ be an even Schwarz function supported in $(-\sigma,\sigma)$. The function $\Psi^\flat_{\phi,X,c,k}(s)$ defined in~\eqref{equation definition Psi holom} is entire and satisfies the bound
	\begin{align*}
		\Psi^\flat_{\phi,X,c,k}(s) \ll_k \frac{X^{\sigma \lvert\Re(s) + \frac{k-1}{2} \rvert}}{(\lvert s \rvert +1)^2 c^{k-1}}.
	\end{align*}
      \end{lemma}

Arguing as in Section~\ref{sec:Primesum} but replacing Proposition~\ref{Prop truncation} by Proposition~\ref{Prop truncation holo} and Lemma~\ref{lemma:bound Psi} by Lemma~\ref{lemma:bound Psi holom2}, we see that Theorem~\ref{theorem main holom} follows once we have shown the following variant of Proposition~\ref{prop:charsumclaim}. 

\begin{proposition}\label{prop:charsumclaim holom}
Let $\varepsilon > 0$ be fixed. Let $k \geq 2$ be an even integer and let $\Theta_k$ be as above. Let $d \geq 2$ be an integer, let $N \geq 2$ and $N_1, \dotsc, N_{40} \geq \frac 12$. For each $j = 1, \dotsc, 40$, let $a_j(n)$ be  $1$, $\mu(n)$ or $\log(n)$, and assume that $a_j(n)$ is $1$ or $\log(n)$ for every $j$ for which $N_j>N^{\frac 1{10}}$. Assume also that
\begin{equation*}
N_1 \dotsm N_{40} \leq N^{ \Theta_k-\varepsilon}.
\end{equation*}
Then
\begin{equation}
\label{eq:charsumclaim holom}
\int_{-\infty}^\infty\sumstar\left|\prod_{j\leq 40}\sum_{n_j\sim N_j}\frac{a_j(n_j)\chi(n_j)}{ n_{j}^{ \frac 12+it}}\right|\frac{\d t}{t^2+1}\ll_k \frac{(N+d)^k N^{1-\frac{k-1}{2}\Theta_k}}{d} (\log(dN))^{O(1)}.
\end{equation}
\end{proposition}
Notice that Proposition~\ref{prop:charsumclaim} is actually the case $k = 2$ of the above proposition. The general case follows from a similar argument. In particular, the proof of Lemma~\ref{le:splitting} can be adapted to show the following result.
\begin{lemma}\label{le:splitting gen}
Let $\eps>0$ be fixed, let $k \geq 2$ be an even integer, let $N \geq 2$ be sufficiently large and let $N_1, \dotsc, N_{40} \geq \frac 12$ be such that
\begin{equation}
\label{eq:Niproductbound2 holom}
N_1 \dotsm N_{40} \leq N^{\Theta_k-\varepsilon}.
\end{equation}
Then at least one of the following holds.
\begin{enumerate}
\item[(a)] There exists a set $I\subseteq\{1,\dots,40\}$ such that 
\begin{equation*}
N^{k\Theta_k - 2k +1 +\frac \eps {100}} < \prod_{j\in I}N_j\leq \frac{N^{\Theta_k}}{N^{k\Theta_k - 2k +1 +\frac \eps {100}}}.
\end{equation*}
\item[(b)] There exist two distinct indices $
j_1,j_2 \in \{1, \dotsc, 40\}$ such that 
\begin{equation*}
\prod_{j \in\{1,\dotsc,40\}\setminus{\{j_1,j_2\}}}N_j\leq N^{k\Theta_k - 2k +1 +\frac \eps {100}}.
\end{equation*}
\end{enumerate}
\end{lemma}

\begin{proof}

We proceed as in the proof of Lemma~\ref{le:splitting}. We let $I$ be a subset of $\{ 1,\dots,40\}  $ containing the maximal number of elements such that 
$P:= \prod_{j\in I} N_j \leq N^{k\Theta_k-2k+1+\frac \eps {100}} $ holds. As in the proof of Lemma~\ref{le:splitting}, $|J| \geq 3$ and, for every $ j\in J = \{ 1,\dots,40 \} \setminus I$,
  \[
    N_j> \frac{N^{\Theta_k}}{P N^{k\Theta_k-2k+1+\frac \eps {100}}},
  \]
 and hence it follows that
\begin{align*}
  \prod_{j = 1}^{40} N_j> P \left(\frac{N^{\Theta_k}}{P N^{k\Theta_k-2k+1+\frac \eps {100}}}\right)^3 \geq \frac{N^{3\Theta_k}}{N^{5(k\Theta_k-2k+1+\frac \eps {100})}} = N^{10k-5 - \Theta_k(5k-3)-\frac{\varepsilon}{20}}.
\end{align*}
To conclude, we note that the above lower bound contradicts~\eqref{eq:Niproductbound2 holom} if the bound
\begin{align*}
  10k-5-\Theta_k(5k-3)-\frac{\varepsilon}{20} > \Theta_k-\varepsilon 
  \end{align*}
holds. Since this is equivalent to 
\begin{align*}
     \Theta_k < 2-\frac{1}{5k-2} + \frac{19}{20(5k-2)}\varepsilon,
\end{align*} the result follows by the definition of $\Theta_k$.
\end{proof}

Now, the proof of Proposition~\ref{prop:charsumclaim holom} proceeds similarly to the proof of Proposition~\ref{prop:charsumclaim}, with slightly different bounds, so we only explain the differences.
Notice first that the right-hand side of~\eqref{eq:charsumclaim holom} satisfies
\begin{align}\label{align:tailoring}
  \begin{aligned}
\frac{(N+d)^k N^{1-\frac{\Theta_k}{2}(k-1)}}{d} &\gg_k \frac{(d^2 N^{k-2}+ d^{\frac 32} N^{k-\frac 32} + dN^{k-1})N^{1-\frac{\Theta_k}{2}(k-1)}}{d} \\
    &= dN^{(k-1)(1-\frac{\Theta_k}{2})} + d^{\frac 12} N^{k-\frac{1}{2}-\frac{\Theta_k}{2}(k-1)} + N^{k-\frac{\Theta_k}{2}(k-1)}. 
  \end{aligned}
  \end{align}
Next, in case Lemma~\ref{le:splitting gen}(a) holds, we obtain the bound
$$\begin{aligned}
	&\int_{-\infty}^\infty\sumstar\left|\prod_{j\leq 40}\sum_{n_j\sim N_j}\frac{a_j(n_j)\chi(n_j)}{ n_{j}^{\frac 12+it}}\right|\frac{\d t}{t^2+1}\\
	&\ll \Bigg( d+d^{\frac 12}\prod_{j\in I}N_j^{\frac 12}+d^{\frac 12}\prod_{j\in J}N_j^{\frac 12}+ \prod_{j =1}^{40} N_j^{\frac 12}\Bigg) (\log(dN))^{O(1)} \\
  &\ll \left(d + d^{\frac 12}\left(\frac{N^{\Theta_k}}{N^{k\Theta_k - 2k +1 +\frac \eps {100}}}\right)^{ \frac 12}+N^{\frac {\Theta_k}2-\frac{\eps}2}\right) (\log(dN))^{O(1)}.
	\end{aligned}
	$$
Since $\Theta_k < 2$, we have $d \leq dN^{(k-1)(1-\frac{\Theta_k}{2})}$
and $N^{\frac{\Theta_k} 2} \leq N^{k-\frac{\Theta_k}{2}(k-1)}$ and thus by~\eqref{align:tailoring} the first and third terms give acceptable contributions to~\eqref{eq:charsumclaim holom}. Furthermore, we have
\[
  d^{\frac 12}\left(\frac{N^{\Theta_k}}{N^{k\Theta_k - 2k +1 +\frac \eps {100}}}\right)^{\frac 12} = d^{ \frac 12} N^{k-\frac{1}{2}-\frac{\Theta_k}{2}(k-1) -\frac{\varepsilon}{200}}
\]
and so, again by~\eqref{align:tailoring}, the second term also gives an acceptable contribution to~\eqref{eq:charsumclaim holom}.

Finally, in case Lemma~\ref{le:splitting gen}(b) holds, we obtain
$$
	\int_{-\infty}^{\infty}\sumstar\left|\prod_{j=1}^{40}\sum_{n_j\sim N_j}\frac{a_j(n_j)\chi(n_j)}{n_j^{\frac 12+it}}\right|\frac{\d t}{t^2+1}
	\ll\bigg( d+\prod_{j\notin\{j_1,j_2\}} N_j\bigg)^{\frac 12} (d+N^{\frac 15})^{\frac 12}(\log(dN))^{O(1)}.
$$
This yields a smaller contribution to~\eqref{eq:charsumclaim holom} than the case using Lemma~\ref{le:splitting gen}(a) and thus the claim in Proposition \ref{prop:charsumclaim holom} follows.

\end{document}